\documentclass[journal]{IEEEtran}
\usepackage{cite}
\usepackage{graphicx,subfigure}
\usepackage{algorithm}
\usepackage{algorithmic}
\usepackage{amsmath,amssymb}
\usepackage[active]{srcltx}

\DeclareMathOperator*{\argmin}{arg\,min}

\def\expect{{\mathbb{E}}}

\newcommand{\BlackBox}{\rule{1.5ex}{1.5ex}}  
\newenvironment{proof}{\par\noindent{\bf Proof\
}}{\hfill\BlackBox\\[2mm]}

\newtheorem{lemma}{\bf{Lemma}}

\newtheorem{corollary}{Corollary}

\newtheorem{remark}{Remark}

\newcommand{\be}{\begin{equation}}
\newcommand{\ee}{\end{equation}}
\newcommand{\bea}{\begin{eqnarray*}}
\newcommand{\eea}{\end{eqnarray*}}

\makeatletter
\def\Lddots{\mathinner{\mkern1mu\raise17\p@\vbox{\kern17\p@\hbox{.}}\mkern2mu
    \raise8\p@\hbox{.}\mkern2mu\raise\p@\hbox{.}\mkern1mu}}
 \makeatother

\outer\def\subsect#1\par{\vskip12pt
plus.07\vsize\penalty-250\vskip0pt plus-.07\vsize
\bigskip\vskip\parskip\message{#1}
\vbox{\smash{\lower9pt\hbox{\kern-8pt\epsfbox{shadedbox.eps}}}}\vskip-\baselineskip
\leftline{\large\bf#1}\nobreak\medskip}

\def\BibTeX{{\rmfamily B\kern-.05em{\scshape i\kern-.025em b}\kern-.08em \TeX}}

\newcommand{\cI}{\mathcal{I}}

\newcommand{\bu}{\mathbf{u}} 
\newcommand{\bA}{\mathbf{A}}

\newcommand{\bF}{\mathbf{F}}
 
\newcommand{\bH}{\mathbf{H}} 
\newcommand{\bI}{\mathbf{I}} 
\newcommand{\bJ}{\mathbf{J}} 
\newcommand{\bM}{\mathbf{M}} 
 
\newcommand{\bQ}{\mathbf{Q}} 
\newcommand{\bR}{\mathbf{R}} 
\newcommand{\bS}{\mathbf{S}}

\newcommand{\bY}{\mathbf{Y}}
\newcommand{\bX}{\mathbf{X}} 
\newcommand{\bW}{\mathbf{W}}

\newcommand{\f}{\mathbf{f}} 
\newcommand{\bzero}{\mathbf{0}} 
\newcommand{\bone}{\mathbf{1}} 
 
\newcommand{\br}{\mathbf{r}} 
\newcommand{\bc}{\mathbf{c}} 
\newcommand{\ba}{\mathbf{a}}
\newcommand{\bee}{\mathbf{e}}

\newcommand{\bm}{\mathbf{m}}

\newcommand{\bb}{\mathbf{b}} 
 
\newcommand{\bx}{\mathbf{x}}

\newcommand{\by}{\mathbf{y}} 
 
\newcommand{\bw}{\mathbf{w}}

\def\expect{{\mathbb{E}}}

\makeatletter
\newcommand*{\coloneq}{\mathrel{\rlap{%
                     \raisebox{0.3ex}{$\m@th\cdot$}}%
                     \raisebox{-0.3ex}{$\m@th\cdot$}}%
                     =}
\makeatother

\begin{document}
\IEEEoverridecommandlockouts
\title{Spatial-Spectral Sensing using the Shrink \& Match Algorithm in Asynchronous MIMO OFDM Signals}
\author{\authorblockN{Saeed Bagheri and Anna Scaglione}\\
\authorblockA{School of Electrical and Computer Engineering\\
University of California, Davis\\
Email: \{sabagheri, ascaglione\}@ucdavis.edu}}
\maketitle
\IEEEpeerreviewmaketitle
\begin{abstract}
Spectrum sensing (SS) in cognitive radio (CR) systems is of paramount importance to approach the capacity limits for the Secondary Users (SU), while ensuring the undisturbed transmission of Primary Users (PU). In this paper, we formulate a cognitive radio (CR)systems spectrum sensing (SS) problem in which Secondary Users
(SU), with multiple receive antennae, sense a channel shared
among multiple asynchronous Primary Users (PU) transmitting
Multiple Input Multiple Output (MIMO) Orthogonal Frequency
Division Multiplexing (OFDM) signals. The method we propose
to estimate the opportunities available to the SUs combines
advances in array processing and compressed channel sensing,
and leverages on both the so called ``shrinkage method'' as well
as on an over-complete basis expansion of the PUs interference
covariance matrix to detect the occupied and idle angles of
arrivals and subcarriers. The covariance ``shrinkage'' step and
the sparse modeling step that follows, allow to resolve ambiguities
that arise when the observations are scarce, reducing the sensing
cost for the SU, thereby increasing its spectrum exploitation
capabilities compared to competing sensing methods. Simulations corroborate that exploiting the sparse representation of the covariance matrix in CR sensing resolves the spatial and frequency spectrum of the sources. 
\end{abstract}

\section{Introduction}
Generally, spectrum-sensing methods include matched filter detection\cite{Kay,MF_1}, likelihood ratio test (LRT) \cite{Kay}, energy detection\cite{ED_1,ED_2,ED_3,ED_4}, and cyclostationary feature detection \cite{CS_1,CS_2,CS_3,CS_4,CS_5}, each of which
has different requirements and advantages/disadvantages.  
The non-coherent energy detector has been shown to be optimal if the cognitive devices have no a priori information about the features of the primary
signals except local noise statistics. In addition, it obviates the need for synchronization with unknown transmitted signals. Matched filter based detection requires a priori knowledge of the primary user, e.g., modulation type, preambles, pilots, pulse shaping, and synchronization of timing and carrier. If the modulation schemes of the primary signals are known, then the cyclostationary feature detector can differentiate primary signals from the local noise by exploiting certain periodicity exhibited by the mean and autocorrelation of the corresponding modulated signals. The cyclostationary detection needs to know the cyclic frequencies of the primary signal, which may not be available to the secondary users in practice. 

These methods detect the presence of the PU within a band and they are presumed to be combined with stochastic control algorithms that decide strategically  what bands the CR receiver should examine next \cite{zhao-magazine}.  Finding spectrum holes in a wideband signal is, instead, the objective of wideband SS  and the focus of this paper. 
The exemplary scenario we envision for the CR is that of a femto-cell access point \cite{femto-cell}, in the role of the SU, searching for spectrum opportunities in a dedicated band with base stations acting as PUs and transmitting MIMO-OFDM signals. OFDM is the modulation of choice for most of the emerging broadband wireless communication physical layer standards. In \cite{Tang} the authors argue that OFDM is the best physical layer candidate for a CR system since it allows to modulate signals so as to fit into discontinuous and arbitrarily wide spectrum segments. OFDM is also optimal from the viewpoint of approaching the Shannon channel capacity in a wideband channel with frequency selectivity and colored noise. Finding the spectrum opportunities in an OFDM system is equivalent to detecting the spectrum holes due to unoccupied subcarriers. A similar problem has been formulated in \cite{OFDM_SS, Quan}. The basic difference in our model is that we do not assume any form of OFDM symbol or frame synchronization among PUs and between PUs and SUs. This reflects the wide practice, for example, in today LTE systems and also allows to tackle more general applications.

Usually, the PUs do not cooperate with the SU and do not transmit specific synchronization signals for the purpose of synchronization at the SU. Synchronization sequences from PUs are going to be sent in each frame to allow their mobile users to synchronize with the frame period. We assume that the CR has lined up to the strongest PU signal frame period and focus instead on the rapid estimation of spectrum holes, via second order methods by estimating the interference covariance matrix. The structure of the covariance matrix is dictated by the cyclic prefix in OFDM symbols, antenna array manifold, occupied subcarriers and channel parameters. In this paper, we focus on the rapid estimation of spectrum holes, via second order methods by estimating the interference covariance matrix. The motivation behind using a second order method is that it is non-coherent, does not require synchronization, and does not need knowledge of the PUs modulation.

Previous papers that focused on wideband SS for CRs, used detectors that leverage the structure of the second order statistics of the PU signal \cite{Axell,SO_1,SO_2,SO_3,SO_4}. Compared to these papers, our MIMO-OFDM sensor approximates a Generalized Likelihood Ratio Test (GLRT) using an estimation algorithm for the covariance that we call {\it Shrink and Match} (S\&M). The S\&M algorithm approximate the maximization of the likelihood function for Gaussian PUs, with respect to the parameters of the PUs signal covariance, by alternating between a step of the {\it shrinkage} algorithm \cite{Chen2} and a step of the Orthogonal Matching Pursuit (OMP) algorithm \cite{Tropp} on an appropriately defined linear sparse model for the interference covariance. Under appropriate conditions, the shrinkage method converges to the ML estimate of the Gaussian PUs covariance \cite{Stein,Tyler} with a relatively low cost iteration, and the OMP algorithm is used to {\it denoise} the estimate obtained at each step, leveraging the sparsity of the model. 

A widely explored trade-off in SS is between the time used for sensing and that used to exploit the channel \cite{Ghasemi}. Thus, CR SS must work with very short data records. Harnessing the benefits of its two steps, the numerical results shown the S\&M method features excellent performance in this regime. 

{\it Notation}: The set of real, complex and integer numbers
numbers by $\mathbb{R}$, $\mathbb{C}$ and $\mathbb{Z}$, respectively. We denote sets by calligraphic symbols,
where the intersection and the union of two sets $\mathcal{A}$ and $\mathcal{B}$
are written as $\mathcal{A}\cap\mathcal{B}$ and $\mathcal{A}\cup\mathcal{B}$, respectively. The operator $|\mathcal{A}|$ on a discrete set takes the cardinality
(measure) of the set and $\mathcal{A}^c$ denotes the complement of $\mathcal{A}$, where the universal set should be evident from the context.  We denote vectors and matrices by boldface lower-case and boldface upper-case symbols. The transpose, conjugate, Hermitian (conjugate) transpose, inverse and pseudo inverse of a matrix $\bX$ are denoted by $\bX^T$, $\bX^*$, $\bX^H$, $\bX^{-1}$ and $\bX^{\dagger}$, respectively. $|\bX|$ and $\text{tr}(\bX)$ denote the determinant and trace of matrix $\bX$, respectively. In this work, the vectorization operator for a matrix is denoted by $\text{vec}(\bX)$. $[\bX]_{a}$ ($[\bx]_{a}$) is the $a$th column (element) of $\bX$ ($\bx$). Similarly, $[\bX]_{\mathcal{A}}$ ($[\bx]_{\mathcal{A}}$) is defined as the collection of columns $[\bX]_a$ (entries $[\bx]_a$) where $a \in \mathcal{A}$. The {\it Frobenius} norm of a matrix is denoted by $\Vert\bX\Vert_F$. The conventional $\ell_2$-norm is written as $\Vert\bx\Vert_2$ and $\Vert\bx\Vert_0$ is the number of non-zero entries of the vector $\bx$. $\otimes$ denotes the Kronecker matrix product and $\odot$ denotes the Hadamard matrix product. The operator $\expect\{\cdot\}$ denotes the expectation operator and a circular symmetric complex Gaussian random vector $\bx \in \mathbb{C}^n$ with mean $\boldsymbol{\mu}\in\mathbb{C}^n$ and covariance matrix $\boldsymbol{\Sigma}\in \mathbb{C}^{n\times n}$ is denoted as $\bx\sim\mathcal{CN}(\boldsymbol{\mu},\boldsymbol{\Sigma})$. $\bone_{v}$ is a $v\times v$ matrix of ones and $\bI_{m}$ is the identity matrix of size $m$. $\bzero_{m\times n}$ denotes an $m\times n$ matrix of zeros. 

\section{System Model}
\label{system model}
In our setting, the set of independent asynchronous multiple antenna PUs is denoted by $\cI$ where $I\triangleq |\cI|$ is the maximum number of active sources. The number of available subcarriers is denoted by $N$ and $L_p$ represents the cyclic prefix length. The length of one OFDM symbol is denoted by $M = N+L_p$. The parameter $T = 1/W$ is the sampling period and $W$ is the bandwidth of the system. Each PU uses a set of subcarriers $\mathcal{C}_i$ for transmission where $\mathcal{C}_i \subseteq \mathcal{C} \triangleq \{0,1,\ldots,N-1\}$. The transmitted stream of discrete samples from the $i$th PU can be described as
\be 
\label{u_i_m}
\bu_{i}[m] = \sum_{k\in \mathbb{Z}}\bu_{i,k}[m-kM], 
\ee 
where $m\in \mathbb{Z}$ and $\bu_{i,k}[m]$ is the $k$th OFDM symbol in the transmitted data stream from the $i$th source and can be represented in matrix form as $\bu_{i,k}[m] = \bA_{i,k}\f'_m$ with $m\in \mathcal{M}\triangleq [0,M)$. The $N_T\times N$ matrix $\bA_{i,k}$ contains the transmitted data symbols in the frequency domain. The vector $\ba_{i,k,c} \triangleq [\bA_{i,k}]_{c+1}\in \mathbb{C}^{N_T \times 1}$ is the modulated data symbols transmitted on subcarrier $c \in \mathcal{C}_i$ in the $k$th OFDM symbol or $\ba_{i,k,c} = \bzero$ if $c \notin \mathcal{C}_i$. In this work, we assume that the vector $\ba_{i,k,c}$ for $c \in \mathcal{C}_i$ is randomly distributed zero mean with covariance matrix $$\expect\left\{\ba_{i,k,c} \ba_{i,k,c}^H\right\}=\frac{1}{N_T|\mathcal{C}_i|}\bI\:.$$ The vector $\f'_m\triangleq [\bF]_{m+1}$ for $m \in \mathcal{M}$ where the $N \times M$ matrix $\bF$ is constructed by appending the last $L_p$ columns of the $N\times N$ IFFT matrix at the beginning.
  
As mentioned before, we consider a MIMO channel where the number of transmit and receive antennas are $N_T$ and $N_R$, respectively. The typical propagation channel in wireless systems is assumed to be a multipath channel with at most $L$ dominant propagation paths from scatterers in the far field. We assume that the antenna spacing is small so that the narrowband array manifold approximation holds. In this case, a general representation for the discrete time MIMO channel between source $i$ and the SU receiver is as follows  
\be 
 \bH_i(m,n) = \sum_{\ell=1}^{L} h_{i,\ell}\:\boldsymbol{\Psi}(m,\boldsymbol{\theta}_{i,\ell})g_i(nT-\tau_{i,\ell})\:,
\ee 
where $h_{i,\ell}$ is the channel fading coefficient with unknown variance. The parameter $\tau_{i,\ell}$ is the delay of $\ell$th path of source $i$ and the matrix $\boldsymbol{\Psi}(m,\boldsymbol{\theta}_{i,\ell})$ is an $N_R\times N_T$ MIMO channel matrix parameterized by $\boldsymbol{\theta}_{i,\ell}$ which also depends on $m$ because of Doppler effects and carrier offsets. The function $g_i(\cdot)$ is the cascade of receive and transmit filter of source $i$.

\begin{figure*}[!ht]
\centerline{\includegraphics[width = 0.7\textwidth]{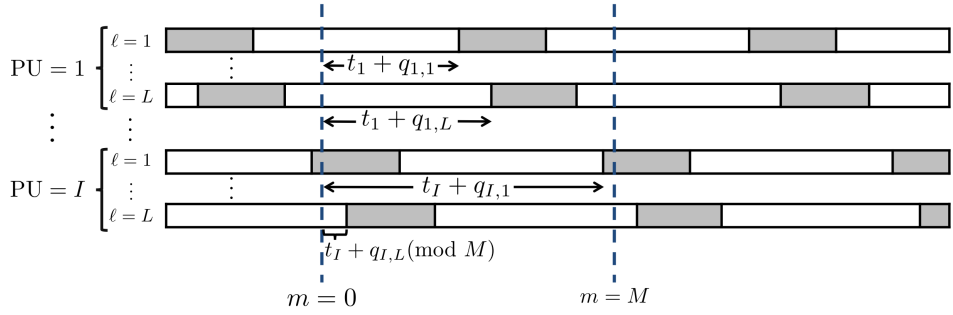}}
\caption{Contribution of different multipath components of received OFDM signals from asynchronous PUs to the observed sequence at the SU.}
\label{fig_ofdm_symbols}
\end{figure*}

\subsection{Signal Model at the Receiver}
The received signal sampled with the rate of $T$ is:
\be
\label{rec_model}
\by[m]=\sum_{i\in {\mathcal I}}\sum_{n} \bH_i(m,n)\bu_i[m-n-t_i]+\bw[m]\:,
\ee
where $\bw[m]\in \mathbb{C}^{N_R}$ is a zero-mean Additive White Gaussian Noise (AWGN) which is both spatially and temporally white and independent of the sources with $\bw[m] \sim \mathcal{CN}(\bzero,\sigma_w^2 \bI_{N_R})$. The unknown parameter $t_i \in \mathcal{M}$ models the misalignment between the received OFDM symbols from the $i$th asynchronous PU and the first observation window (starting at $m=0$) at the SU (See Fig. \ref{fig_ofdm_symbols}). 

In order to simplify the receiver model, following the approach in \cite{Simon_LA,matt}, we discretize the parameters in our channel model with a certain resolution. We quantize the delay and the parameter vector $\boldsymbol{\theta}_{i,\ell}$ and denote the quantized values by
\begin{align}
Q(\tau_{i,\ell}) \triangleq q_{i,\ell}T~, ~~~ \bar{\boldsymbol{\theta}}_{i,\ell} \triangleq \bQ_{\Theta}(\boldsymbol{\theta}_{i,\ell}) \in \mathcal{A}_{\boldsymbol{\Theta}}\:.
\end{align} 
In the following derivations, we ignore the model mismatch error. The parameter $T$ is the time resolution of the time quantization grid and the integer value $q_{i,\ell} \in \mathcal{Q} = [0,L_p)$ is the index of the discrete delay. $\mathcal{A}_{\boldsymbol{\Theta}}$ represents the finite set of quantized parameter vectors, whose cardinality is $n_{\boldsymbol{\Theta}}\triangleq |\mathcal{A}_{\boldsymbol{\Theta}}|$. The function $\bQ_{\Theta}(\cdot)$ is the quantizer associated with the parameter vector $\boldsymbol{\theta}_{i,\ell}$ which can be explicitly described as 
\be 
\bQ_{\Theta}(\boldsymbol{\theta}_{i,\ell}) = \argmin_{\boldsymbol{\theta} \in \mathcal{A}_{\boldsymbol{\Theta}}} \left\Vert\boldsymbol{\theta}_{i,\ell} - \boldsymbol{\theta}\right\Vert_2^2\:.
\ee

The discrete-time received signal is hereinafter modeled as: 
\begin{align} 
\label{y_m}
\by[m] &= \sum_{i\in {\mathcal I}}\sum_{\ell=1}^{L} \sum_{k \in \mathbb{Z}} h_{i,\ell}^k\: \boldsymbol{\Psi}(m,\bar{\boldsymbol{\theta}}_{i,\ell})\bu_{i,k}[m-kM-q_{i,\ell}-t_i] \nonumber \\&+ \bw[m]\:,
\end{align}
where the channel coefficients are independent for different $i$, $\ell$ and $k$: $$\expect\{h_{i,\ell}^{k}(h_{i',\ell'}^{k'})^*\} = \sigma_{i,\ell}\:\delta[i-i']\delta[\ell-\ell']\delta[k-k']\:.$$

\section{Sparse Covariance Matrix Representation}
In this section, we present a sparse model for the PUs MIMO-OFDM interference covariance matrix for the asynchronous model in \eqref{y_m}. In order to accurately estimate the covariance matrix, we require multiple ($K$) independent and identically distributed (i.i.d.) samples of the received signal with the same covariance structure, which means the parameters $\{ \sigma_{i,\ell},\bar{\boldsymbol{\theta}}_{i,\ell},q_{i,\ell},t_i\}_{\ell=1,\ldots,L}^{i=1,\ldots,I}$ remain the same in the time required to collect $K$ samples. To ensure that these observation samples are i.i.d., the SU should use a sample shift equal to $2M$ between consecutive collected sequences to guarantee that firstly all collected data sequences share the same set of parameters $\{t_i\}_{i=1}^{I}$ and secondly they are uncorrelated and independent. Thus, the $r$th $N_R M\times 1$ observation sample in the SU is expressed as
\be 
\by_r = \begin{bmatrix}
 \by[r2M] \\
 \by[r2M+1]\\
 \vdots \\
 \by[r2M+M-1] 
\end{bmatrix}.
\ee 

The following Lemma summarizes the assumptions and the mathematical sparse model for the covariance matrix which depends on the unknown channel parameters and the unknown set of occupied subcarriers.
\begin{lemma}
\label{lemma1}
{\it Assume that $\expect\{\ba_{i,k,c}\ba_{i,k,c}^H\} =  \frac{1}{N_T|\mathcal{C}_i|}\bI$ for $c\in \mathcal{C}_i$ and the transmitted symbols are zero mean and independent for different $i$, $k$ and $c$. The received signal is zero-mean and the sparse representation of its covariance matrix is expressed as }
\begin{align}
\label{Sigma_mat_general}
\text{vec}(\boldsymbol{\Sigma}(\boldsymbol{\sigma},\sigma_w^2)) &\triangleq \text{vec}(\expect\{\by_r\by_r^H\})= \bM\boldsymbol{\sigma} + \sigma_w^2\text{vec}(\bI_{MN_R})\:,
\end{align}
{\it where $\boldsymbol{\sigma}$ is an $MNn_{\boldsymbol{\Theta}}\times 1$ sparse coefficient vector with non-negative entries. The $(MN_R)^2\times MNn_{\boldsymbol{\Theta}}$ matrix $\bM$ is the over-complete dictionary and its columns are constructed with $\text{vec}(\boldsymbol{\Pi}_{v,\boldsymbol{\theta},c}^0 + \boldsymbol{\Pi}_{v,\boldsymbol{\theta},c}^1)$ for  $v\in \mathcal{M}$, $\boldsymbol{\theta}\in\mathcal{A}_{\boldsymbol{\Theta}}$ and $c \in \mathcal{C}$. The parameter $v$ capture the unknown delays and misalignments, $c$ denotes the subcarriers index and $\boldsymbol{\theta}$ models all the possible remaining unknown channel parameters. The matrix $\boldsymbol{\Pi}_{v,\boldsymbol{\theta},c}^j$ for $j=0,1$ is defined as
\begin{align}
\label{Pi_0}
\boldsymbol{\Pi}_{v,\boldsymbol{\theta},c}^j &\triangleq (\bJ^j(v)\otimes\bI)\mathcal{T}(\{\boldsymbol{\Psi}(m,\boldsymbol{\theta})\}_{m=0}^{M-1})(\f_c\f_c^H\otimes \bI)\nonumber \\ 
&~~~\times\mathcal{T}(\{\boldsymbol{\Psi}^H(m,\boldsymbol{\theta})\}_{m=0}^{M-1})((\bJ^j(v))^H\otimes\bI)~,
\end{align}
where $\f_c = [\bF^T]_{c+1}$. $\bJ^j(v)$ for $j=0,1$ is the shift matrix defined as 
\begin{align}
\bJ^0(v) &= 
\begin{bmatrix}
 \bzero_{v \times M-v} &  \bI_{v} \\
 \bzero_{M - v} &  \bzero_{M - v \times v} 
\end{bmatrix}, \bJ^1(v) = 
\begin{bmatrix}
 \bzero_{v \times M-v} &   \bzero_{v}\\
 \bI_{M - v}  &  \bzero_{M - v \times v} 
\end{bmatrix}\nonumber
\end{align}
The operator $\mathcal{T}(\{\mathbf{A}_m\}_{m=0}^{M-1})$ creates a block diagonal matrix with blocks $\mathbf{A}_0,\ldots, \mathbf{A}_{M-1}$ that appear in its input argument.}
\end{lemma}
\begin{proof}
See Appendix A.
\end{proof}

It is worth mentioning that the sparse model in \eqref{Sigma_mat_general} for the covariance matrix satisfying the constraint $\boldsymbol{\sigma} \geq 0$, guarantees that $\boldsymbol{\Sigma}(\boldsymbol{\sigma},\sigma_w^2)$ is Hermitian positive semi-definite. Next, we present our proposed method for estimating the structured covariance matrix in \eqref{Sigma_mat_general}.

\section{The Shrink and Match Algorithm}
\label{est_cov_mat}
In the CR application, one important objective of the SU is to decrease the sensing time in order to increase the time to exploit the empty subcarriers for transmission. As a result, we focus on the cases where the number of observations $K$ is small and the sample covariance matrix $\widehat{\bS} \triangleq \frac{1}{K}\sum_{r=0}^{K-1}\by_r\by_r^H$ cannot be considered a good estimate of the true covariance matrix. 
Our proposed method relies on finding first a suboptimal solution that approximates the ML estimate for the case of Gaussian PUs symbols and then on matching it to the sparse model in Lemma 1. More specifically, the $d\times K$ matrix $\bY \triangleq [\by_{0},\ldots,\by_{K-1}]$ is the matrix of collected zero-mean i.i.d. samples with the same covariance structure where, in general, $d$ denotes the dimension of the problem. Even if the PUs are not Gaussian sources, and the ML receiver would suggest to perform a complex joint detection, $\by_r$ can be approximated to be complex Gaussian distributed, as a maximum entropy approximation of its distribution. The likelihood function for the  unknown parameters $\boldsymbol{\sigma}$ and $\sigma_w^2$ given the observation $\bY$ is expressed as 
\begin{align}
\label{ML_opt}
&\mathcal{L}(\boldsymbol{\sigma},\sigma_w^2|\bY) = \frac{\exp\left[-\text{tr}\left(\bY\bY^{H}\boldsymbol{\Sigma}^{-1}(\boldsymbol{\sigma},\sigma_w^2)\right)\right]}{\pi^{Kd}|\boldsymbol{\Sigma}(\boldsymbol{\sigma},\sigma_w^2)|^K}\:.
\end{align}
The sparsity regularized log-likelihood function of \eqref{ML_opt} is 
\begin{align}
\label{H0_opt}
  \widehat{\boldsymbol{\sigma}},\widehat{\sigma}_{w}^2 &\triangleq \argmin_{\boldsymbol{\sigma},\sigma_w^2} \log|\boldsymbol{\Sigma}(\boldsymbol{\sigma},\sigma_w^2)| + \text{tr}\left[\boldsymbol{\Sigma}^{-1}(\boldsymbol{\sigma},\sigma_w^2)~ \widehat{\bS}\right] + {\kappa} \Vert \boldsymbol{\sigma} \Vert_0 \nonumber \\
 & \text{s.t. }~~~ \sigma_w^2 \geq 0, ~~\boldsymbol{\sigma} \geq 0
\end{align}
where $\Vert \boldsymbol{\sigma} \Vert_0$ is the imposed sparsity constraint and ${\kappa}$ is some regularization parameter. The ML estimate of the covariance matrix is given by $\boldsymbol{\Sigma}(\widehat{\boldsymbol{\sigma}},\widehat{\sigma}_{w}^2)$.

In \eqref{H0_opt}, we have relaxed the original problem by imposing the sparsity constraint. However, even solving this relaxed problem, which is non-convex, is in general complicated. In the following, in order to simplify the optimization, we assume that the noise variance $\sigma_w^2$ has been estimated and is known.



Under the small sample size constraint $K < d$, we approximate the solution of \eqref{H0_opt} by first obtaining an accurate estimate of the covariance matrix which is invertible and then projecting the estimate of the
covariance matrix in the sparse model in \eqref{Sigma_mat_general}. We call this method the Shrink \& Match (S\&M) algorithm. The two steps are explained next. 

\subsubsection{Shrinkage Step}
For estimating the covariance matrix, we use the method proposed in \cite{Chen2}. In \cite{Chen2}, the covariance estimation is based on shrinkage regularized fixed point iterations for a high dimensional setting, where the fixed point iterations converge to the ML estimate. However, in \cite{Chen2}, the structure of the covariance matrix has been ignored. 

This shrinkage method works with trace-normalized covariance matrices ($\text{tr}(\boldsymbol{\Sigma}) = d$) and in this work, the final estimate must be scaled by the estimated value of the trace. Thus, we need to estimate the trace of the covariance matrix which is obtained as 
\be 
\widehat{\text{tr}}(\boldsymbol{\Sigma})\triangleq \frac{1}{K}\sum_{r=0}^{K-1}\Vert\by_r\Vert_2^2\:.
\ee 
The shrinkage coefficient is estimated as \cite{Chen2}   
\be
\hat{\gamma} = \left\{ \begin{array}{ll}
\dfrac{d^2-\dfrac{1}{d}\:\text{tr}(\widehat{\bR}\widehat{\bR}^H)}{d^2-Kd - K+\Big(K+\dfrac{(K-1)}{d}\Big)\text{tr}(\widehat{\bR}\widehat{\bR}^H)}, & K<d \\
0  & K\geq d
\end{array}
\right. 
\ee
where $\widehat{\bR} \triangleq \frac{d}{K}\sum_{r=0}^{K-1}\by_r\by_r^H/\Vert\by_r\Vert_2^2$ is the trace-normalized sample covariance matrix. The iterative steps update the covariance matrix estimate at each iteration as follows
\begin{align}
\label{Sigma_tilde}
\widetilde{\boldsymbol{\Sigma}}_{j+1}&= (1 - \hat{\gamma})\dfrac{d}{K}\sum_{r=0}^{K-1}\dfrac{\by_r\by_r^H}{\by_r^H\widehat{\boldsymbol{\Sigma}}_{j}^{-1}\by_r} + \hat{\gamma}\bI\:.
\end{align}
At this step, we normalize the estimated covariance matrix as
\be 
\widehat{\boldsymbol{\Sigma}}_{j+1} = \dfrac{d~\widetilde{\boldsymbol{\Sigma}}_{j+1}}{\text{tr}(\widetilde{\boldsymbol{\Sigma}}_{j+1})}\:,
\ee
so that it is trace-normalized and return to the update step \eqref{Sigma_tilde} until the stopping criterion is met.

\subsubsection{Matching Step}
In order to impose the structure of the covariance matrix, after the convergence of the iterative process, we first scale the estimated covariance matrix (as $\widetilde{\boldsymbol{\Sigma}} = \frac{1}{d}\widehat{\text{Tr}}(\boldsymbol{\Sigma})\widehat{\boldsymbol{\Sigma}}_{j+1}$ ) to compensate for the trace-normalized assumption, and then fit the covariance matrix with the model in \eqref{Sigma_mat_general} by minimizing the following cost
\begin{align} 
\label{sigma_tilde_j}
\widehat{\boldsymbol{\sigma}} & \triangleq \argmin_{\boldsymbol{\sigma}}\left\Vert\text{vec}(\widetilde{\boldsymbol{\Sigma}} - {\sigma}_w^2\bI )- {\bM} \boldsymbol{\sigma}\right\Vert_{2}^2  + \lambda\Vert\boldsymbol{\sigma}\Vert_0\nonumber \\
& \text{s.t.}~~~ \boldsymbol{\sigma} \geq 0
\end{align}

\begin{remark}
The S\&M algorithm solves a problem similar to the asymptotic ML (AML) estimator for a structured (Hermitian Toeplitz) covariance matrix posed in \cite{Li_1999}, however we have two main differences. In \cite{Li_1999}, the authors use the inverse of the sample covariance matrix to define a weighted  $\ell_2$-norm that they use to estimate the parameters in their model, an approach that is close to optimum in the asymptotic regime of $K \gg d$. Experimentally, we have observed that the choice of the shrinkage method to estimate the covariance matrix, followed by our sparse denoising step using the $\ell_2$-norm instead of the weighted norm, work better when $K<d$. 
\end{remark}

The problem in \eqref{sigma_tilde_j} can be solved by using greedy methods and in particular the non-negative OMP algorithm \cite{non_OMP} in order to satisfy the non-negativity constraint on the variables. It can also be relaxed by imposing the sparsity regularized constraint $\Vert\cdot\Vert_1$, and in this case, can be solved by convex programming as a linear program. Algorithm \ref{alg_H0} summarizes the steps required for this suboptimal method to approximate the solution of \eqref{H0_opt}. Algorithm \ref{NN-OMP} summarizes the steps of the Non-negative OMP method to solve \eqref{sigma_tilde_j}.

\begin{algorithm}
\caption{S\&M Algorithm}
\begin{algorithmic}[1] \label{alg_H0}
\REQUIRE $\bY$, ${\bM}$ and $\sigma_w^2$ 
\STATE {\bf Initialize}: $j=0$ and $\widehat{\boldsymbol{\Sigma}}_{0} = \bI_{d}$. Calculate $\widehat{\text{tr}}(\boldsymbol{\Sigma})$ and $\hat{\gamma}$.
\STATE {\bf Repeat}
\STATE Calculate $\widetilde{\boldsymbol{\Sigma}}_{j+1}$ from \eqref{Sigma_tilde}.
\STATE {\bf Normalization}: $\widehat{\boldsymbol{\Sigma}}_{j+1} = \dfrac{d~\widetilde{\boldsymbol{\Sigma}}_{j+1}}{\text{tr}(\widetilde{\boldsymbol{\Sigma}}_{j+1})}.$
\STATE set $j = j+1$
\STATE {\bf Until} { Stopping criterion}: $\Vert\widehat{\boldsymbol{\Sigma}}_{j}-\widehat{\boldsymbol{\Sigma}}_{j-1}\Vert_{F}^{2}\leq \tau_{\text{min}}\Vert\widehat{\boldsymbol{\Sigma}}_{j-1}\Vert_{F}^{2}$ 
\STATE Solve \eqref{sigma_tilde_j} using Algorithm \ref{NN-OMP} to find $\widehat{\boldsymbol{\sigma}}$.
\end{algorithmic}
\end{algorithm}

\begin{algorithm}
\caption{Non-negative OMP Algorithm}
\begin{algorithmic}[1] \label{NN-OMP}
\REQUIRE obtain $\bM = [\bm_1,\bm_2,\ldots,\bm_D]$, where $D\triangleq MNn_{\boldsymbol{\Theta}}$ and  $\bb\triangleq \text{vec}(\widetilde{\boldsymbol{\Sigma}} - {\sigma}_w^2\bI )$ 
\STATE {\bf Initialize}: $\boldsymbol{\sigma}_0=\bzero$, $\mathcal{S}_0 = \emptyset$,  $\br_0 = \bb$ and $j=1$.
\STATE {\bf Repeat}
\STATE compute 
$$\epsilon(d) = \Vert\br_{j-1}\Vert_2^2 - \dfrac{(\max\{\bm_d^T\br_{j-1},0\})^2}{\Vert\bm_{d}\Vert_2^2}\:,$$
for $1\leq d\leq D$.
\STATE $d^* = \argmin_{d \in \mathcal{S}_{j-1}^c} \epsilon(d)$
\STATE {\bf Update Support}: $\mathcal{S}_{j} = \mathcal{S}_{j-1} \cup \{d^*\}$.
\STATE $\bM_j = [\bM]_{\mathcal{S}_{j}}$ 
\STATE {\bf Find the Non-negative Solution}: 
\begin{align}
\boldsymbol{\sigma}_j^* &= \argmin_{\boldsymbol{\sigma}} \Vert\bM_j\boldsymbol{\sigma} - \bb\Vert_2^2\nonumber \\
& \text{s.t.}~~~ \boldsymbol{\sigma} \geq 0 \nonumber
\end{align}
\STATE {\bf Update Solution}:
\begin{align}
\left[\boldsymbol{\sigma}_j\right]_{\mathcal{S}_{j}} &= \boldsymbol{\sigma}_j^*\:, \nonumber \\
\left[\boldsymbol{\sigma}_j\right]_{\mathcal{S}_{j}^c} &= \bzero \:
\end{align}
\STATE {\bf Update Residual}: $\br_j = \bb - \bM\boldsymbol{\sigma}_j$.
\STATE set $j = j+1$
\STATE {\bf Until} { Stopping criterion}: $\Vert{\boldsymbol{\sigma}}_{j}-{\boldsymbol{\sigma}}_{j-1}\Vert_{2}^{2}\leq \tau_{\text{OMP}}\Vert{\boldsymbol{\sigma}}_{j-1}\Vert_{2}^{2}$ 
\end{algorithmic}
\end{algorithm}

\section{The Separable S\&M Algorithm}
\label{sec:est_subcarrier}
While the channel model used so far is valid for various array configurations and models for time variations, in this section we focus on the more familiar case of uniform linear arrays (ULA) for the receive and transmit antennas and a single Doppler per path. The reason to focus on this case is because the covariance has a separable structure that allows to introduce a greatly simplified version of the S\&M Algorithm.

The discrete time channel between the SU receiver and the $i$th PU can be described \cite{mimo_ch,mimo_ch2} by 
\be 
\label{ula_channel_model}
\boldsymbol{\Psi}(m,\boldsymbol{\theta}_{i,\ell}) = e^{j2\pi \psi_{i,\ell}m T}\bee_r(\beta_{i,\ell})\bee_t^{H}(\alpha_{i,\ell})\:,
\ee
where $\boldsymbol{\theta}_{i,\ell} = [\psi_{i,\ell},\beta_{i,\ell},\alpha_{i,\ell}]$. The vector $\bee_t(\alpha_{i,\ell}) = [1,e^{j2\pi\alpha_{i,\ell}},\ldots,e^{j2\pi(N_T-1)\alpha_{i,\ell}}]^T$ is the steering vector associated with the angle of departure (AoD) and $\bee_r(\beta_{i,\ell}) = [1,e^{j2\pi\beta_{i,\ell}},\ldots,e^{j2\pi(N_R-1)\beta_{i,\ell}}]^T$ is the steering vector associated with the angle of arrival (AoA), where the parameters $\alpha_{i,\ell}$ and $\beta_{i,\ell}$ model the angles of departure and arrival of the $\ell$th propagation path between the SU and the $i$th PU, respectively. In this model, the frequency offset is incorporated into the Doppler spread of the channel. Thus, the parameter $\psi_{i,\ell}\triangleq f_i + \omega_{i,\ell}$ models both carrier frequency offset denoted by $f_i$ and Doppler effects of the $\ell$th path of the $i$th PU denoted by $\omega_{i,\ell}$ with $\psi_{\text{max}}\triangleq \max_{i,\ell} \psi_{i,\ell}$ 

The quantized counterpart of $\boldsymbol{\theta}_{i,\ell}$ is
\be 
\bar{\boldsymbol{\theta}}_{i,\ell} = \Big[(\dfrac{p_{i,\ell}}{P_I})\Delta f, \dfrac{b_{i,\ell}}{B}, \dfrac{a_{i,\ell}}{A}\Big]\:,
\ee
where the parameter $\Delta f \triangleq {1}/{NT}$ denotes the subcarrier frequency. The frequency offset is modelled as ${p_{i,\ell}}/{P_I}$ to capture the fractions of the subcarrier frequency (with the resolution of $\Delta f/P_I$) where the integer value $p_{i,\ell} \in \mathcal{P} \triangleq [p_l,p_l+P)$ where $p_l$ is the smallest value in the set and $|\mathcal{P}|=P$. $1/A$ and $1/B$ are the angular resolution of the angle of departure and arrival, respectively. Moreover, the integer parameters $a_{i,\ell}\in \mathcal{A} \triangleq[0,A)$ and $b_{i,\ell}\in \mathcal{B} \triangleq [0,B)$. 

\begin{corollary}
\label{cor-1}
{\it For the quantized discrete time channel model 
$$\boldsymbol{\Psi}(m,\bar{\boldsymbol{\theta}}_{i,\ell}) = e^{j2\pi \frac{p_{i,\ell}m}{P_IN}}\bee_r(b_{i,\ell}/B)\bee_t^{H}(a_{i,\ell}/A)\:,$$
the $(MN_R)^2\times MNPB$ over-complete dictionary $\bM$ is constructed as 
\be 
[\bM]_{\eta(v,p,c,b)} = \text{vec}((\boldsymbol{\Psi}_{p,c}\odot \boldsymbol{\Upsilon}(v))\otimes (\bee_r(b/B)\bee_r^H(b/B)))\:,
\ee 
where $\eta(v,p,c,b) \triangleq v PNB + (p-p_l)NB + cB + b + 1$. In this model, $v\in \mathcal{M}$, the parameter $p\in \mathcal{P}$ models the Doppler, $c \in \mathcal{C}$ and $b\in \mathcal{B}$. The matrix $\boldsymbol{\Upsilon}(v)$ denotes a masking matrix and is defined as 
\begin{align}
\label{define_mask_mat}
\boldsymbol{\Upsilon}(v) &\triangleq 
\begin{bmatrix}
 \bone_{v}& \bzero_{v \times M-v}    \\
 \bzero_{M - v \times v} &  \bone_{M - v }
\end{bmatrix}.
\end{align} 
The matrix $\boldsymbol{\Psi}_{p,c}\triangleq \boldsymbol{\Lambda}(p)\f_c\f_c^H\boldsymbol{\Lambda}^H(p)$ represents the dependency of the dictionary components on subcarrier $c$ and discretized Doppler $p$ where $\boldsymbol{\Lambda}(p) \triangleq \text{diag}\{e^{j2\pi \frac{p m}{P_I N}}\}_{m=0}^{M-1}$ is a diagonal matrix containing the Doppler coefficients. The vector $\boldsymbol{\sigma}$ in the sparse model \eqref{Sigma_mat_general} is a $MNPB\times 1$ sparse coefficient vector with non-negative entries.}
\end{corollary}
\begin{proof}
See Appendix B. 
\end{proof}

In the following, we explicitly explain how the S\&M algorithm can be simplified in this scenario. In the special case where $N_R \geq IL$, it is more convenient to divide the problem into two separate smaller and easier to solve covariance estimation problems, in the spatial and temporal domains reducing the memory required to store the dictionary elements, the complexity of the algorithm, decreasing its runtime. This is possible because of the far field narrowband approximation in the channel model, which allows to approximately separate the effects of path delays and of the different arrival times at each array element. We call this method, {\it Separable Shrink and Match} (SS\&M) algorithm. In the first step, the SU estimates the AoAs, exploiting the existing structure in the spatial covariance matrix that depends on the parameters $\{b_{i,\ell}\}_{\ell=1,\ldots,L}^{i=1,\ldots,I}$. Then using the AoAs information, the SU, for each angle, filters spatially the temporal streams to produce observation free from interference from the other active directions. Then, it uses the temporal sparse covariance matrix representation to recover the occupied subcarriers. The mathematical formulation of these two steps is given in the following.

{\it Step 1}: To estimate the spatial covariance matrix $\boldsymbol{\Sigma}_S \triangleq \expect\{\by[m]\by[m]^H\}$, the SU samples the received signal at the antenna array at $m = kM$ and collect $K$ i.i.d. observations $\bY_S\triangleq[\by[0],\by[M],\ldots,\by[(K-1)M]]$ with the same spatial covariance matrix. The sparse representation of the spatial covariance matrix can be expressed as $$\text{vec}(\boldsymbol{\Sigma}_S) = \bM_S\boldsymbol{\sigma}_S + \sigma_w^2\text{vec}(\bI_{N_R})\:,$$ where $\boldsymbol{\sigma}_S$ is a $B\times 1$ sparse coefficient vector with non-negative entries. The $N_R^2\times B$ spatial dictionary matrix is constructed as $[\bM_S]_{b+1}=\text{vec}(\bee_r(b/B)\bee_r^H(b/B))$ for $b\in \mathcal{B}$.

{\it Step 2}: Similarly, to estimate the temporal covariance matrix $\boldsymbol{\Sigma}_T$, the receiver collects $K$ observation blocks $\bY_T(r)\triangleq [\by[2rM],\ldots,\by[2rM + M -1]]$, $r\in\{0,1,\ldots,K-1\}$ of size $N_R\times M$ where $2M$ is the amount of sample shift between consecutive collected sequences. The temporal covariance matrix of the spatially filtered observations $\widetilde{\bY}_T(r)\triangleq \boldsymbol{\Phi}^{\dagger}{\bY}_T(r)$ is expressed as $\boldsymbol{\Sigma}_{T,l} \triangleq \expect\{\tilde{\by}_{r,l}^T\tilde{\by}_{r,l}^*\}$ where $\tilde{\by}_{r, l}$ is the $l$th row of $\widetilde{\bY}_T(r)$. $\boldsymbol{\Phi}$ is the subspace of AoAs and $\boldsymbol{\Phi}^{\dagger}$ denotes its pseudo inverse. The sparse representation of the temporal covariance matrix is $$\text{vec}(\boldsymbol{\Sigma}_{T,l}) =\bM_T\boldsymbol{\sigma}_{T,l} + \sigma_w^2\Vert\boldsymbol{\phi}_l^T\Vert_2^2\text{vec}(\bI_M)\:,$$ where $\boldsymbol{\sigma}_{T,l}$ is a $MNP\times 1$ sparse coefficient vector with non-negative entries. The $M^2\times MNP$ temporal dictionary matrix is constructed as $[\bM_T]_{\mu(v,p,c)} = \text{vec}(\boldsymbol{\Psi}_{p,c}\odot \boldsymbol{\Upsilon}(v))$ where $\mu(v,p,c) \triangleq v PN + (p-p_l)N + c + 1$. 

For the derivations of the sparse representation of the spatial and temporal covariance matrices see Appendices C and D, respectively. Algorithm \ref{alg2} summarizes the two steps of the SS\&M algorithm.

\begin{algorithm}
\caption{SS\&M Algorithm}
\begin{algorithmic}[1] \label{alg2}
\REQUIRE $\bY_S$, ${\bM}_{S}$, $\{\bY_T(r)\}_{r=0}^{K-1}$, ${\bM}_{T}$ and $\sigma_w^2$
\STATE Apply Algorithm \ref{alg_H0} with ${\bM}_{S}$ and $\bY_S$ and find the non-negative sparse vector $\widehat{\boldsymbol{\sigma}}_S$.
\STATE {\bf Set of detected AoAs}: $\mathcal{A}_{\phi} \triangleq \{b: [\widehat{\boldsymbol{\sigma}}_S]_{b+1} > 0, b\in \mathcal{B} \}$.
\STATE {\bf Subspace of AoAs}: $\boldsymbol{\Phi}\triangleq \{\bee_r(b/B)\}_{b \in \mathcal{A}_{\phi}}$.
\STATE Compute $\boldsymbol{\Phi}^{\dagger}$ the pseudo inverse of $\boldsymbol{\Phi}$
\STATE $\widetilde{\bY}_T(r)\triangleq \boldsymbol{\Phi}^{\dagger}{\bY}_T(r)$ for $r=0,\ldots,K-1$.
\STATE {\bf for} $l=1,\ldots, |\mathcal{A}_{\phi}|$ {\bf do}
\STATE Apply Algorithm \ref{alg_H0} with ${\bM}_{T}$ and $\bY_{l} \triangleq [\tilde{\by}_{0,l}^T, \ldots,\tilde{\by}_{K-1,l}^T]$, and find the sparse non-negative vector $\widehat{\boldsymbol{\sigma}}_{T,l}$
\STATE $\mathcal{E}_{l} \triangleq \{c: [\widehat{\boldsymbol{\sigma}}_{T,l}]_{\mu(v,p,c)} > 0, v \in \mathcal{M}, p \in \mathcal{P}, c\in \mathcal{C} \}$
\STATE {\bf end for}
\STATE {\bf Set of occupied subcarriers}: ${\mathcal{E}} = \bigcup_{l=1}^{|\mathcal{A}_{\phi}|} \mathcal{E}_{l}$
\end{algorithmic}
\end{algorithm}

\section{Numerical Experiments}
\label{num_exp}
In this section, we examine the performance of our proposed algorithm numerically. We consider a MIMO OFDM system with $I = 4$ uncorrelated asynchronous sources where $N_T=2$ is the number of transmit antennas. The other channel and system parameters are considered to be $N=64$, $L_p=8$ and $L = 2$. We model the first arrival of the OFDM symbols from each user as a uniform random variable in $\mathcal{M}$ and the delays are distributed uniformly within the cyclic prefix duration. The Doppler $p_{i,\ell}$ of each path is generated as a uniform discrete random variable in the set $\mathcal{P}$ with $P = 3$. The AoAs and AoDs are continuous values in $[0,2\pi]$ where the angular difference between the AoAs is greater than $10^{\circ}$. In the simulations, we have set the number of grid points $A = B = 180$. Thus, the angle resolution of $1$ degree. We generate uncorrelated Rayleigh fading coefficients, $h_{i,\ell}^k \sim \mathcal{CN}(0,1)$. The same transmission power equal to $1$ is considered for all the sources. Throughout this entire section, the signal to noise ratio (SNR) is defined as $-10\log_{10} \sigma_w^2$ where the noise is zero-mean AWGN with variance equal to $\sigma_w^2$. In Algorithm \ref{alg_H0}, the value of $\tau_{\text{min}}$ is set to be $10^{-4}$.

In order to demonstrate the capability of SS\&M to make efficient use of the empty subcarriers without causing harmful interference to the PUs, we define two metrics. In the following, probability of false alarm and missed detection in the $c$th subcarrier are denoted by $P_{\text{FA}}(c)$ and $P_{\text{MD}}(c)$, respectively. The opportunistic spectral utilization of subcarrier $c$ is measured with $1-P_{\text{FA}}(c)$. As a result, the aggregate opportunistic throughput of the SU can be described as 
$$\sum_{c=0}^{N-1}R_c(1 -P_{\text{FA}}(c)),$$
where $R_c$ denotes the throughput achievable over the $c$th subcarrier if used by the SU. Assuming that $R_c$ is the same for all $c \in \mathcal{C}$, the expected aggregate opportunistic throughput of the SU can be measured by $$\rho_{T} \triangleq 1 - \dfrac{1}{N}\sum_{c=0}^{N-1}P_{\text{FA}}(c).$$ Assuming that all the PUs are equally important, the aggregate interference to PUs can be expressed as $$\sum_{c=0}^{N-1}C_cP_{\text{MD}}(c),$$ where $C_c$ denotes the cost incurred by causing interference with a PU in the $c$th subcarrier. Assuming that $C_c$ is equal for $c\in \mathcal{C}$, the average aggregate interference to PUs can be equivalently measured by $$\rho_I \triangleq \sum_{c=0}^{N-1}P_{\text{MD}}(c).$$

Ideally, we want to have $\rho_T$ close to $1$ while at the same time guarantee small values for $\rho_{I}$ below some threshold $\epsilon$. In Fig. \ref{fig:fig_3}, $\rho_{T}$ and $\rho_{I}$ have been plotted versus SNR when $N_R = 12$ for $K=20$ and $K=30$ for the same simulation runs. The curves are obtained by computing $P_{\text{MD}}$ and $P_{\text{FA}}$ for each subcarrier $c$ numerically in $100$ independent simulation runs. This figure illustrates that SS\&M exploits the empty subcarriers very efficiently (high aggregate throughput close to $1$) while causing small interference to the primary users (relatively small values for $\rho_I$) even for small sample sizes. In addition, we can observe that increasing the SNR or $K$ does not increase $\rho_T$ dramatically. However, increasing $K$ or SNR decreases $\rho_I$. This trade-off is a critical aspect of this algorithm that one should consider to choose $K$ based on the available SNR and the existing constraint on $\rho_I < \epsilon$.

\begin{figure}[!ht]
\centering
\includegraphics[width=1\linewidth]{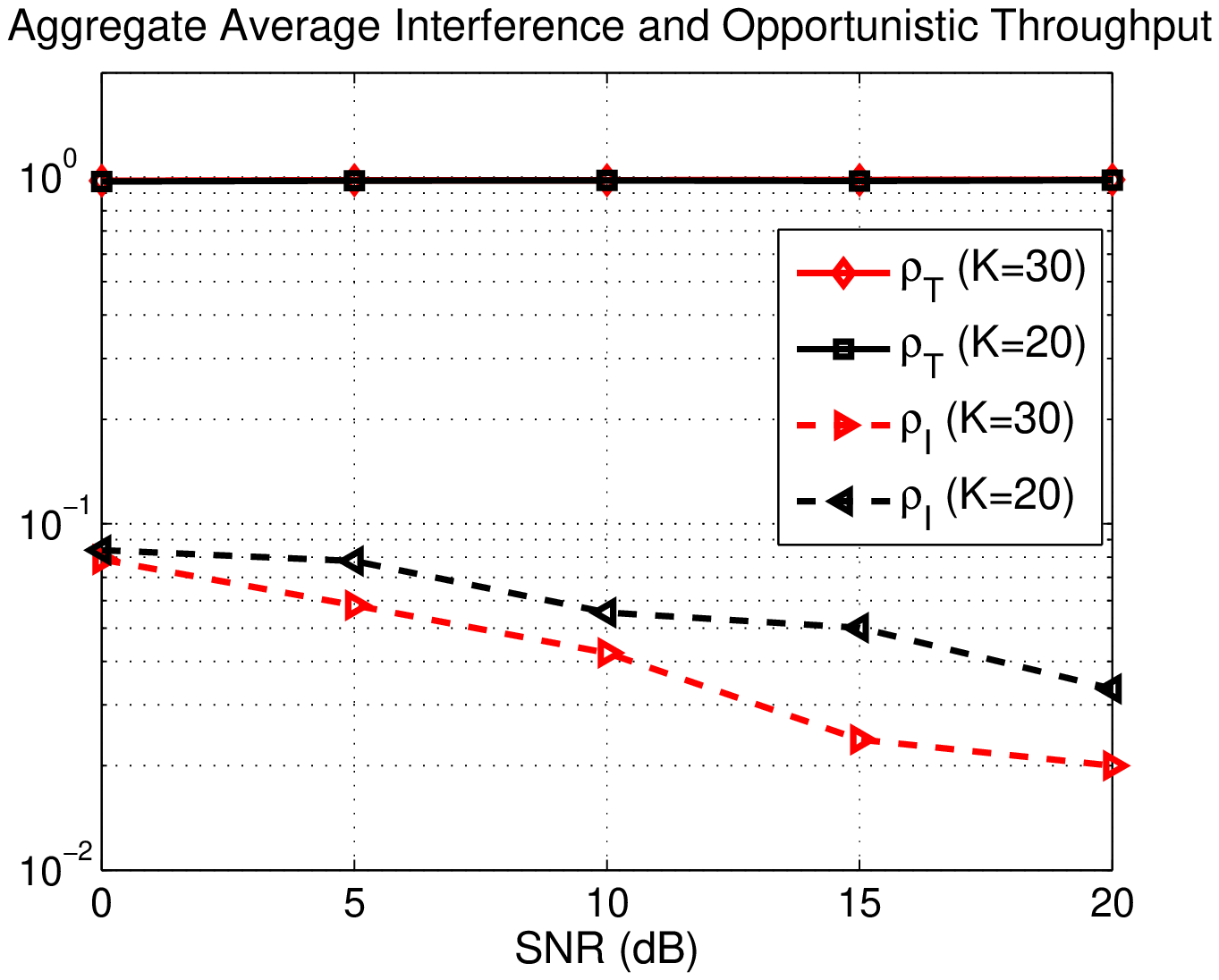}
\caption{$\rho_T$ and $\rho_I$ versus SNR when $N_R=12$ for $K = 20$ and $K=30$.}
\label{fig:fig_3}
\end{figure}

To the best of our knowledge, in the literature, there is not any method that can be fairly compared with S\&M in our asynchronous setting. As a result, we compare its performance with existing methods in estimating the covariance matrix $\boldsymbol{\Sigma}(\widehat{\boldsymbol{\sigma}})$ and AoAs. To test the MSE of our proposed covariance estimator S\&M, we compare its performance with the sample covariance, shrinkage method in \cite{Chen2} and shrinkage MMSE \cite{Chen1} estimates. For these tests, we use the temporal covariance matrix with $d=M=72$. For all simulations, we set $\text{SNR} = 0 \text{ dB}$, $|\mathcal{C}_i|=6$ for all $i\in \cI$ and let $K$ range from $5$ to $50$. Fig. \ref{fig:mse} shows the normalized MSE of the estimators defined as $$\dfrac{\Vert\boldsymbol{\Sigma} - \widehat{\boldsymbol{\Sigma}}\Vert_F^2}{\Vert\boldsymbol{\Sigma}\Vert_F^2},$$ for different values of $K$, where $\widehat{\boldsymbol{\Sigma}}$ denotes any of the covariance matrix estimates. It is evident that S\&M outperforms all the methods and in addition, it is very robust when $K$ varies and even for $K=20$ an acceptable performance has been achieved.



\begin{figure}[!ht]
\centering
\includegraphics[width=1\linewidth]{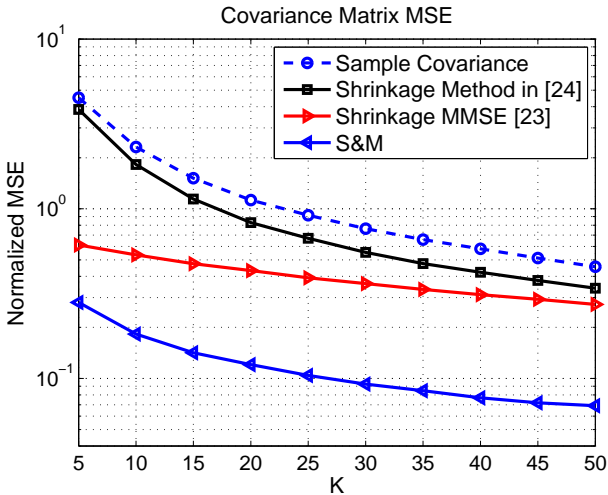}
\caption{MSE of covariance estimators for $100$ channel realizations.}
\label{fig:mse}
\end{figure}

\begin{figure}[!ht]
\centering
\includegraphics[width=1\linewidth]{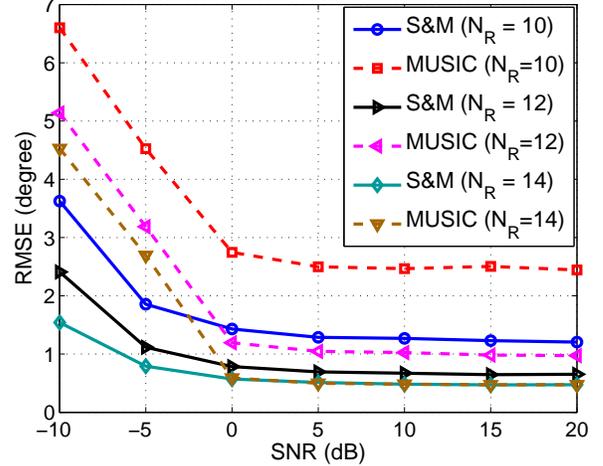}
\caption{AoA estimation RMSE when $K=20$ for $N_R = 10,12,14$.}
\label{fig:fig_1}
\end{figure}

In Fig. \ref{fig:fig_1}, the performance of S\&M in estimating the AoA is presented in terms of root mean square error (RMSE) and compared with root-MUSIC (Multiple Signal Classification) algorithm\footnote{Even though they are not shown here, we have compared the performance with standard MUSIC and ESPRIT. However, the root-MUSIC algorithm had the best performance for AoA estimation among the methods based on unstructured covariance matrix estimation.}. The curves are obtained by averaging the results of $1000$ independent simulation runs for different SNR values and $N_R$. The RMSE decreases by SNR and the gap between the root-MUSIC and S\&M RMSE reduces by increasing the number of antennas. However, in low SNR values, S\&M outperforms root-MUSIC even for large number of antennas.

\section{Conclusions}
\label{conclusions}
We proposed a new algorithm for spatial and spectral sensing in asynchronous MIMO-OFDM signals. Our method first finds an accurate estimate of the covariance matrix of the received signal at the SU using shrinkage method for small sample size. Then, it matches the estimate with a sparse representation of the covariance matrix which depends on the channel parameters and occupied subcarriers. The SU finally uses the support of the estimated sparse coefficient vector to recover the spatial and spectral pattern of the received signal.

\appendices

\section{Proof of Lemma \ref{lemma1}}
\label{app_1}

In the asynchronous scenario, at most two consecutive OFDM symbols transmitted from the $i$th PU are captured in each multipath contribution of $\by_r$, $r=0,\ldots,K-1$. The $N_RM \times 1$ vector of received signal for any of the $K$ collected observations is written as $\by_r = \bx_r + \bw_r$, where $\bx_r$ is the received signal due to the transmitted signals from the PUs and $\bw_r$ is the noise. The vector of received signal corresponding to the transmitted signals $\bx_r$ can be written as follows 
\be 
\label{x_r_vec}
\bx_r = \sum_{i\in {\mathcal I}}\sum_{\ell=1}^{L} \bx_r^0(i,\ell)+ \bx_r^1(i,\ell),
\ee 
where at the $r$th observation period the contribution from the $\ell$-th path of the $i$-th user $\bx_r(i,\ell)$ is written as the sum of the contributions of two consecutively transmitted OFDM symbols. The vector $\bx_r^1({i,\ell})$ corresponds to the OFDM symbol which its beginning is captured and $\bx_r^0({i,\ell})$ denotes the previous OFDM symbol which its tail is being observed.
 
Each of the vectors $\bx_r^j({i,\ell}) ,~~ j=0,1$ can be written in terms of the channel parameters and the random transmitted data symbols as
\be 
\label{x_r_j_vec}
\bx_r^j(i,\ell) = h_{i,\ell}^{k_j}(\bJ^j(v_{i,\ell})\otimes\bI)\mathcal{T}(\{\boldsymbol{\Psi}(m,\bar{\boldsymbol{\theta}}_{i,\ell})\}_{m=0}^{M-1})(\bF^T\otimes \bI)\ba_{i}^{k_j}
\ee 
where the indices $k_j, ~j=0,1$ are the indices of the observed OFDM symbols at the $r$th observation sequence corresponding to the $\ell$th multipath component of $i$th PU. They can be described as follows where for simplicity we have dropped their dependency on $i$, $\ell$ and $r$ in their symbol names
\be 
k_0 = 2r-\zeta - 1, ~~~ k_1 = 2r-\zeta,
\ee 
where $\zeta = \left\lfloor\dfrac{t_i + q_{i,\ell}}{M}\right\rfloor$. The parameter $v_{i,\ell} \in \mathcal{M}$ is defined as $v_{i,\ell}\triangleq (t_i + q_{i,\ell}) (\text{mod } M)$ and represents the relative displacement between the $\ell$-th multipath component of the observed OFDM symbols transmitted from $i$-th user and the start of the observation window at the SU (see Fig. \ref{fig_ofdm_symbols}). The vector $\ba_{i}^{k_j} \triangleq \text{vec}(\bA_{i}^{k_j})$, $j=0,1$ is the vector of data symbols of the $i$th source in the frequency domain where non-zero rows in $\bA_i^{k_j}$ correspond to the active subcarriers of the $i$th source. The matrix $\bJ^j(v_{i,\ell})$ is the shift matrix and is defined as 
\begin{align}
\label{define_shift_mat}
\bJ^0(v) &= 
\begin{bmatrix}
 \bzero_{v \times M-v} &  \bI_{v} \\
 \bzero_{M - v} &  \bzero_{M - v \times v} 
\end{bmatrix},\nonumber \\
\bJ^1(v) &= 
\begin{bmatrix}
 \bzero_{v \times M-v} &   \bzero_{v}\\
 \bI_{M - v}  &  \bzero_{M - v \times v} 
\end{bmatrix}.
\end{align}

The covariance matrix of the received signal at the SU is defined as $\boldsymbol{\Sigma} \triangleq \expect\{\by_r \by_r^H\}$. Since, the transmitted data and noise are independent, we can write 
$$\boldsymbol{\Sigma} = \boldsymbol{\Sigma}_X  +  \sigma_w^2\bI,$$
where $\boldsymbol{\Sigma}_X  = \expect\{\bx_r\bx_r^H\}$. In addition, $\ba_{i}^{k_0}$ and $\ba_{i}^{k_1}$ are independent for every $i$ and by exploiting the fact that $\expect\{h_{i,\ell}^{k}(h_{i',\ell'}^{k'})^*\} = \sigma_{i,\ell}\delta(i-i')\delta(\ell-\ell')\delta(k-k')$, we can write $\boldsymbol{\Sigma}_X = \boldsymbol{\Sigma}_{X}^0 + \boldsymbol{\Sigma}_{X}^1$ where
\be 
\boldsymbol{\Sigma}_{X}^j \triangleq \sum_{i \in \cI}\sum_{\ell =1}^{L}\expect\{\bx_r^j({i,\ell}) (\bx_r^j({i,\ell}))^H\}.
\ee 

For notational convenience, we introduce the coefficient
\begin{align}
\label{alpha}
\alpha_{i,v,\boldsymbol{\theta}}^j &= \sum_{u=1}^{I}\sum_{\ell=1}^{L} h_{u,\ell}^{k_j}~\delta_{\boldsymbol{\theta},\bar{\boldsymbol{\theta}}_{u,\ell}}\delta[v-v_{u,\ell}]\delta[i-u],
\end{align}
where $v \in \mathcal{M}$ and $\boldsymbol{\theta} \in \mathcal{A}_{\boldsymbol{\Theta}}$ to indicate whether the $i$th source is being observed and whether there exists a link at a certain delay plus misalignment $vT$ with a certain channel parameter $\boldsymbol{\theta}$. The indicator $\delta_{\bc,\bar{\bc}}$ is defined as
\be 
\delta_{\bc,\bar{\bc}}\triangleq  \left\{ \begin{array}{ll}
1, & \bc=\bar{\bc} \\
0  & \text{otherwise}
\end{array}
\right. 
\ee 
Using the parameterization introduced in \eqref{alpha}, it follows that
\begin{align}
\bx_r^j(i,v,\boldsymbol{\theta}) &= \alpha_{i,v,\boldsymbol{\theta}}^j\: (\bJ^j(v)\otimes\bI)(\mathcal{T}(\{\boldsymbol{\Psi}(m,\boldsymbol{\theta})\}_{m=0}^{M-1}))\nonumber \\ & ~~~\times (\bF^T\otimes \bI)\ba_{i}^{k_j}.
\end{align}

We define $\sigma_{i,v,\boldsymbol{\theta}} \triangleq \expect\{\alpha_{i,v,\boldsymbol{\theta}}^j(\alpha_{i,v,\boldsymbol{\theta}}^j)^*\}$ as the variance of the channel coefficient which in our setting does not depend on $j$ and remains the same for all $K$ collected observations. The expression for $\boldsymbol{\Sigma}_{X}^j$ taking to account the parametrization introduced in \eqref{alpha} is rewritten as follows
\begin{align}
\label{Sigma_X_j}
\boldsymbol{\Sigma}_{X}^j &\triangleq\sum_{i \in \cI}\sum_{v,\boldsymbol{\theta}} \expect\{\bx_r^j(i,v,\boldsymbol{\theta}) (\bx_r^j(i,v,\boldsymbol{\theta}))^H\}\nonumber\\
&= \sum_{i \in \cI}\sum_{v=0}^{M-1}\sum_{\boldsymbol{\theta}\in \mathcal{A}_{\boldsymbol{\Theta}}} \sigma_{i,v,\boldsymbol{\theta}} ~\boldsymbol{\Pi}_{v,\boldsymbol{\theta}}^j~,
\end{align}
where 
\begin{align}
&\boldsymbol{\Pi}_{v,\boldsymbol{\theta}}^j \triangleq\nonumber \\ &(\bJ^j(v)\otimes\bI)\mathcal{T}(\{\boldsymbol{\Psi}(m,\boldsymbol{\theta})\}_{m=0}^{M-1})(\bF^T\text{diag}(\{\gamma_{i,c}\}_{c=0}^{N-1})\bF^*\otimes \bI)\nonumber \\ 
~~~&\times\mathcal{T}(\{\boldsymbol{\Psi}^H(m,\boldsymbol{\theta})\}_{m=0}^{M-1})((\bJ^j(v))^H\otimes\bI),
\end{align}
where we have replaced $\expect\Big\{\ba_{i}^{k_j}(\ba_{i}^{k_j})^H \Big\}=\text{diag}(\{\gamma_{i,c}\}_{c=0}^{N-1})\otimes\bI_{N_T}$ in the expression. The coefficients $\gamma_{i,c}$ models whether the subcarrier $c$ is occupied or not and is defined as
\be 
\label{gamma_ic}
\gamma_{i,c}\triangleq  \left\{ \begin{array}{ll}
\dfrac{1}{N_T|\mathcal{C}_i|}, & c \in \mathcal{C}_i \\
0  & \text{otherwise}
\end{array}
\right. 
\ee 

The term $\bF^T\text{diag}(\{\gamma_{i,c}\}_{c=0}^{N-1})\bF^*$ is simplified as 
\be 
\bF^T\text{diag}(\{\gamma_{i,c}\}_{c=0}^{N-1})\bF^* = \sum_{c=0}^{N-1}\gamma_{i,c}\:\f_c\f_c^H
\ee 
where $\f_c = [\bF^T]_{c+1}$. Then, we can rewrite \eqref{Sigma_X_j} as
\be 
\boldsymbol{\Sigma}_{X}^j = \sum_{i \in \cI}\sum_{v=0}^{M-1}\sum_{\boldsymbol{\theta}\in \mathcal{A}_{\boldsymbol{\Theta}}} \sum_{c=0}^{N-1} \sigma_{i,v,\boldsymbol{\theta}}\:\gamma_{i,c}\:\boldsymbol{\Pi}_{v,\boldsymbol{\theta},c}^j~,
\ee 
where $\boldsymbol{\Pi}_{v,\boldsymbol{\theta},c}^j$ is redefined as
\begin{align}
\label{Pi}
\boldsymbol{\Pi}_{v,\boldsymbol{\theta},c}^j &\triangleq (\bJ^j(v)\otimes\bI)\mathcal{T}(\{\boldsymbol{\Psi}(m,\boldsymbol{\theta})\}_{m=0}^{M-1})(\f_c\f_c^H\otimes \bI)\nonumber \\ 
&~~~\times\mathcal{T}(\{\boldsymbol{\Psi}^H(m,\boldsymbol{\theta})\}_{m=0}^{M-1})((\bJ^j(v))^H\otimes\bI).
\end{align}

The expression in \eqref{Pi} illustrates that we can define the dictionary components as $\text{vec}(\boldsymbol{\Pi}_{v,\boldsymbol{\theta},c}^0 + \boldsymbol{\Pi}_{v,\boldsymbol{\theta},c}^1)$ which do not depend on $i$. Thus, as far as the estimation of the covariance matrix is concerned, we can define the new coefficients 
\be 
\label{sigma_v}
\sigma_{v,\boldsymbol{\theta},c} \triangleq \sum_{i \in \cI} \sigma_{i,v,\boldsymbol{\theta}}\:\gamma_{i,c}~,
\ee 
by summing over $i$. Using the vectorization operator, the sparse representation of the covariance matrix is more compactly written as 
\begin{align}
\label{Sigma_z_mat}
\text{vec}(\boldsymbol{\Sigma}) &= \bM\boldsymbol{\sigma} + \sigma_w^2\text{vec}(\bI_{MN_R}),
\end{align}
where $\bM$ is the over-complete dictionary and its columns are constructed with $\text{vec}(\boldsymbol{\Pi}_{v,\boldsymbol{\theta},c}^0 + \boldsymbol{\Pi}_{v,\boldsymbol{\theta},c}^1)$ for  $v\in \mathcal{M}$, $\boldsymbol{\theta}\in\mathcal{A}_{\boldsymbol{\Theta}}$ and $c \in \mathcal{C}$. The vector $\boldsymbol{\sigma}$ is an $MNn_{\boldsymbol{\Theta}}\times 1$ sparse coefficient vector with non-negative entries equal to $\sigma_{v,\boldsymbol{\theta},c}$.

\section{Proof of Corollary \ref{cor-1}}
In this case, $\boldsymbol{\theta}_{i,\ell} = [p_{i,\ell},b_{i,\ell},a_{i,\ell}]$, the corresponding discretized vector is denoted by $\boldsymbol{\theta} = [p,b,a]$ where $p\in \mathcal{P}$, $a\in \mathcal{A}$, $b\in \mathcal{B}$ and $\mathcal{A}_{\boldsymbol{\Theta}}=\mathcal{P}\times\mathcal{B}\times\mathcal{A}$. 

In order to explicitly describe the dictionary components, we replace $\boldsymbol{\Psi}(m,\boldsymbol{\theta})$ in \eqref{Pi} with $e^{j2\pi \frac{p m}{P_I N_0}}\bee_r(b/B) \bee_t^{H}(a/A)$. Then, it follows that
\be 
\label{T_psi}
\mathcal{T}(\{\boldsymbol{\Psi}(m,\boldsymbol{\theta})\}_{m=0}^{M-1}) = \boldsymbol{\Lambda}(p) \otimes \bee_r(b/B) \bee_t^{H}(a/A),
\ee 
where $\boldsymbol{\Lambda}(p) \triangleq \text{diag}\left\{e^{j2\pi \frac{p m}{P_I N_0}}\right\}_{m=0}^{M-1}$ is a diagonal matrix containing the Doppler coefficients.

substituting \eqref{T_psi} in \eqref{Pi}, results in the following expression for $\boldsymbol{\Pi}_{v,p,b,a,c}^j$
\begin{align}
\label{Pi_mod}
\boldsymbol{\Pi}_{v,p,b,a,c}^j &\triangleq (\bJ^j(v)\otimes\bI)\Big(\boldsymbol{\Lambda}(p)\f_c\f_c^H\boldsymbol{\Lambda}^H(p)\otimes \bee_r(b/B) \bee_t^{H}(a/A)\nonumber \\ &~~~\times\bee_t(a/A)\bee_r^H(b/B)\Big)((\bJ^j(v))^H\otimes\bI)\nonumber \\
&= N_T\Big(\bJ^j(v)\boldsymbol{\Lambda}(p)\f_c\f_c^H\boldsymbol{\Lambda}^H(p)(\bJ^j(v))^H\Big) \nonumber \\ & ~~~ \otimes \bee_r(b/B)\bee_r^H(b/B),
\end{align}
which does not depend on $a$.

We define the matrix $\boldsymbol{\Xi}_{i,v,p,c}^j$ which does not depend on the parameter $a$ anymore as
\begin{align}
\label{XI_ic}
\boldsymbol{\Xi}_{i,v,p,c}^j \triangleq  N_T\Big(\bJ^j(v)\boldsymbol{\Lambda}(p)\f_c\f_c^H\boldsymbol{\Lambda}^H(p)(\bJ^j(v))^H\Big).
\end{align}
By the entrywise computation of \eqref{XI_ic}, one can verify that the sum $\boldsymbol{\Xi}_{i,v,p,c}^0+\boldsymbol{\Xi}_{i,v,p,c}^1$ is simplified to
\begin{align}
\label{XI_ic_sim}
\boldsymbol{\Xi}_{i,v,p,c}^0 + \boldsymbol{\Xi}_{i,v,p,c}^1 = N_T\Big(\boldsymbol{\Lambda}(p)\f_c\f_c^H\boldsymbol{\Lambda}^H(p)\Big) \odot \boldsymbol{\Upsilon}(v),
\end{align}
where $\boldsymbol{\Upsilon}(v)$ denotes a masking matrix and is defined in \eqref{define_mask_mat}. 

The covariance matrix is eventually expressed as 
\begin{align}
\label{Sigma_z_mod}
\boldsymbol{\Sigma}_{X} &= \sum_{i \in \cI}\sum_{v ,p,b,c}N_T~\sigma_{i,v,p,b}\:\gamma_{i,c}\:\big( \boldsymbol{\Psi}_{p,c} \odot \boldsymbol{\Upsilon}(v)\big) \nonumber \\ &~~~ \otimes  \big(\bee_r(b/B)\bee_r^H(b/B)\big),
\end{align}
where $\boldsymbol{\Psi}_{p,c} \triangleq \boldsymbol{\Lambda}(p)\f_c\f_c^H\boldsymbol{\Lambda}^H(p)$.
The expression in \eqref{Sigma_z_mod} illustrates that we can define the dictionary elements as $\boldsymbol{\Psi}_{p,c} \odot \boldsymbol{\Upsilon}(\theta)$ in the temporal domain and $\bee_r(b/B)\bee_r^H(b/B)$ in the spatial domain which do not depend on $i$. Similar to \eqref{sigma_v}, we define the new coefficients 
\be 
\sigma_{v,p,b,c} \triangleq \sum_{i \in \cI} N_T\:\sigma_{i,v,p,b}\:\gamma_{i,c}.
\ee 

In matrix form, the sparse representation of the covariance matrix is similar to \eqref{Sigma_z_mat}. The results of Corollary \ref{cor-1} follows where the vector $\boldsymbol{\sigma}$ is a $MNPB\times 1$ sparse coefficient vector with non-negative entries defined as $[\boldsymbol{\sigma}]_{\eta(v,p,c,b)} \triangleq \sigma_{v,p,b,c}$.

\section{Derivation of the Spatial Covariance Matrix}
The $N_R\times N_R$ covariance matrix in the spatial domain is defined as $\boldsymbol{\Sigma}_{S} \triangleq \expect\{\by[m] \by^H[m]\} = \expect\{\bx[m]\bx^H[m] \} + \sigma_w^2\bI$. We first find the covariance of the $m$th and the $m'$th received vectors at the SU as
\begin{align}
\label{cov_mm}
&\expect\{\bx[m]\bx^H[m'] \} = \sum_{i\in {\mathcal I}}\sum_{\ell=1}^{L} \sum_{k,k' \in \mathbb{Z}}\sigma_{i,\ell}\: e^{j2\pi \frac{p_{i,\ell} (m-m')}{P_I N}} \bee_r(b_{i,\ell}/B)\nonumber \\ 
&~~~\times \bee^{H}_r(b_{i,\ell}/B)\expect\Big\{\bee_t^{H}(a_{i,\ell}/A)\bA_{i,k}\f'_{m-kM-q_{i,\ell}-t_i}\nonumber \\ 
&~~~\times \f'^H_{m'-k'M-q_{i,\ell}-t_{i}}\bA_{i,k'}^H \bee_t(a_{i,\ell}/A)\Big\}.
\end{align}
The term $\bee_t^{H}(a_{i,\ell}/A)\bA_{i,k}\f'_{m-kM-q_{i,\ell}-t_i}$ in the expectation in \eqref{cov_mm} is a scalar. As a result, one can rewrite the expectation as 
\begin{align}
&\f'^H_{m'-k'M-q_{i,\ell}-t_{i}}\expect\Big\{\bA_{i,k'}^H \bee_t(a_{i,\ell}/A)\bee_t^{H}(a_{i,\ell}/A)\bA_{i,k}\Big\}\nonumber \\ &~~~\times\f'_{m-kM-q_{i,\ell}-t_i}.\nonumber
\end{align}
By exploiting the independence of the rows of $\bA_{i,k'}$ we have
\begin{align}
\expect&\Big\{\bA_{i,k'}^H \bee_t(a_{i,\ell}/A)\bee_t^{H}(a_{i,\ell}/A)\bA_{i,k}\Big\} \nonumber \\
&~~~~~~= N_T\:\text{diag}(\{\gamma_{i,c}\}_{c=0}^{N-1})\delta[k-k'],
\end{align}
where $\gamma_{i,c}$ is defined in \eqref{gamma_ic}. Thus, one can rewrite \eqref{cov_mm} as
\begin{align}
\label{Sigma_X}
&\expect\{\bx[m]\bx^H[m'] \} = \sum_{i\in {\mathcal I}}\sum_{\ell=1}^{L} \sum_{k \in \mathbb{Z}} N_T\:\sigma_{i,\ell}\: e^{j2\pi \frac{p_{i,\ell} (m-m')}{P_I N}}\nonumber \\
&\times \f'^H_{m'-kM-q_{i,\ell}-t_{i}}\text{diag}(\{\gamma_{i,c}\}_{c=0}^{N-1})\f'_{m-kM-q_{i,\ell}-t_i}\nonumber \\
&\times \bee_r(b_{i,\ell}/B)\bee^{H}_r(b_{i,\ell}/B).
\end{align}
The expression in \eqref{Sigma_X} demonstrates that the $m$th and the $m'$th observations are uncorrelated if $m' = m + kM$ for $k\in \mathbb{Z}$ and thus to collect K i.i.d. samples one has to choose $m=kM$ for $k=0,\ldots, K-1$. 
Specifically, when $m=m'$, we can rewrite \eqref{Sigma_X} as
\begin{align}
\expect\{\bx[m]\bx^H[m] \} = \sum_{i\in {\mathcal I}}\sum_{\ell=1}^{L} \:\sigma_{i,\ell}\: \bee_r\Big(\dfrac{b_{i,\ell}}{B}\Big)\bee^{H}_r\Big(\dfrac{b_{i,\ell}}{B}\Big).
\end{align}

In order to express the covariance matrix over an overcomplete basis, we introduce the coefficient
\begin{align}
\vartheta_{i,b} &= \sum_{j=1}^{I}\sum_{\ell=1}^{L}\:\sigma_{j,\ell}\:\delta[b-b_{j,\ell}]\delta[i-j].
\end{align}
Then, we can express $\boldsymbol{\Sigma}_{S}$ equivalently as
\begin{align}
\boldsymbol{\Sigma}_{S} = \sum_{b =0}^{B-1} \xi_b~ \bee_r(b/B)\bee_r^H(b/B)  + \sigma_w^2\bI ,
\end{align} 
where $\xi_b \triangleq \sum_{i \in \cI}\vartheta_{i,b}$. In matrix form, the sparse representation of the spatial covariance matrix is more compactly written as 
\begin{align}
\label{Sigma_S_mat}
\text{vec}(\boldsymbol{\Sigma}_{S}) &= \bM_{S}\boldsymbol{\sigma}_S + \sigma_w^2\text{vec}(\bI),
\end{align}
where $[\bM_{S}]_{b+1} = \text{vec}(\bee_r(b/B)\bee_r^H(b/B))$ and $[\boldsymbol{\sigma}_S]_{b+1} \triangleq \xi_b$.

\section{Derivation of the Temporal Covariance Matrix}
The $M\times M$ temporal covariance matrix for the $l$th set of spatially filtered observations $\widetilde{\bY}_T(r)\triangleq \boldsymbol{\Phi}^{\dagger}{\bY}_T(r)$ in the temporal domain is expressed as $\boldsymbol{\Sigma}_{T,l} \triangleq \expect\{\tilde{\by}_{r,l}^T\tilde{\by}_{r,l}^*\}$ where $\tilde{\by}_{r, l}$ is the $l$th row of $\widetilde{\bY}_T(r)$. In the asynchronous scenario, at most two consecutive OFDM symbols transmitted from the $i$th source are captured in $\tilde{\by}_{r,l} = \boldsymbol{\phi}_{l}\bY_T(r)$ where $\boldsymbol{\phi}_{l}$ is the $l$th row of $\boldsymbol{\Phi}^{\dagger}$ and $\bY_T(r) = \bX_r + \bW_r$. The $N_R\times M$ block of received signal from the sources for any of the $K$ collected observations can be written similar to \eqref{x_r_vec} as follows
\begin{align}
\bX_r & = \sum_{i \in \cI}\sum_{\ell =1}^{L} \bX_r^0({i,\ell}) + \bX_r^1({i,\ell}),
\end{align}
where $\text{vec}(\bX_r^j(i,\ell)) = \bx_r^j(i,\ell)$ for $j=0,1$. 

For notational convenience, similar to \eqref{alpha}, we introduce the coefficient
\begin{align}
\label{alpha_new}
\alpha_{i,v,p,a,b}^j &= \sum_{u=1}^{I}\sum_{\ell=1}^{L}h_{u,\ell}^{k_j}\:\delta[a-a_{u,\ell}]\delta[b-b_{u,\ell}]\nonumber \\
& ~~~\times \delta[v-v_{u,\ell}]\delta[p-p_{u,\ell}]\delta[i-u],
\end{align}
where $v \in \mathcal{M}$, $a \in \mathcal{A}$ and $b\in \mathcal{B}$.

%

Using the parametrization introduced in \eqref{alpha_new}, one can express $\bX_r^j({i,\ell})$ as
\be 
\bX_r^j({i,v,p,a,b}) = \alpha_{i,v,p,a,b}^j \bee_r(b/B) \bee_t^{H}(a/A) \bS^j(i,v,p),
\ee 
where $\bS^j(i,v,p)$ is defined as follows
\be 
\bS^j(i,v,p) = \bA_{i}^{k_j} \bF \boldsymbol{\Lambda}(p)(\bJ^j(v))^T.
\ee

The temporal covariance matrix $\boldsymbol{\Sigma}_{T,l} \triangleq \expect\{\by_{r,l}^T \by_{r,l}^*\} = \expect\{\bX_r^T\boldsymbol{\phi}_l^T \boldsymbol{\phi}_l^* \bX_r^*\} + \sigma_w^2\Vert\boldsymbol{\phi}_l^T\Vert_2^2\bI$ can be equivalently expressed as  $\boldsymbol{\Sigma}_{T,l} = \boldsymbol{\Sigma}_{T,l}^0 + \boldsymbol{\Sigma}_{T,l}^1 + \sigma_w^2\Vert\boldsymbol{\phi}_l^T\Vert_2^2\bI$ where 
\begin{align}
\label{Sigma_temporal}
\boldsymbol{\Sigma}_{T,l}^j &\triangleq \sum_{i \in \cI}\sum_{v ,p,a,b} \expect\{(\bX_r^j({i,v,p,a,b}))^T(\bX_r^j({i,v,p,a,b}))^*\}\nonumber \\
&= \sum_{i \in \cI}\sum_{v ,p,b,c} \sigma_{i,v,p,a,b}~\pi_{b,l}~\gamma_{i,c}~\boldsymbol{\Xi}_{i,v,p,c}^j\:,
\end{align}
and $\sigma_{i,v,p,a,b} \triangleq \expect\{\alpha_{i,v,p,a,b}^j(\alpha_{i,v,p,a,b}^j)^*\}$ is the variance of the channel coefficient which does not depend on $j$ and is the same for all $K$ collected observations. The matrix $\boldsymbol{\Xi}_{i,v,p,c}^j$ is defined in \eqref{XI_ic}. In \eqref{Sigma_temporal}, we have replaced 
\begin{align}
\expect&\{(\bA_{i}^{k_j})^T\bee_t^{*}(a/A)\bee_r^T(b/B)\boldsymbol{\phi}_l^T \boldsymbol{\phi}_l^*\bee_r^*(b/B)\bee_t^{T}(a/A)\bA_{i}^{k_j}\} \nonumber \\ 
&~~~= N_T~\pi_{b,l}~\text{diag}(\{\gamma_{i,c}\}_{c=0}^{N-1})\:,
\end{align}
where $\expect\{(\bA_{i}^{k_j})^T\bee_t^{*}(a/A)\bee_t^{T}(a/A)\bA_{i}^{k_j}\} = \text{diag}(\{\gamma_{i,c}\}_{c=0}^{N-1})$ and the parameter $\pi_{b,l} \triangleq |\bee_r^T(b/B)\boldsymbol{\phi}_l^T|^2$ is a scalar adding more sparsity to our model, because it is equal to zero when $b \in \mathcal{A}_{\phi}$. 

The expression in \eqref{XI_ic} illustrates that the dictionary elements, $\big(\boldsymbol{\Xi}_{\theta,p,c}^0 + \boldsymbol{\Xi}_{\theta,p,c}^1)$, expressed in \eqref{XI_ic_sim}, do not depend on $i$ and $a$. Thus, as far as the estimation of the temporal covariance matrix and the occupied subcarriers are concerned, we can remove $a$ from the indices and define the new coefficients 
\be 
\label{sigma_l_vpc}
\sigma_{v,p,c}^l \triangleq \sum_{i \in \cI}\sum_{b=0}^{B-1} \sigma_{i,v,p,b}~\pi_{b,l}~\gamma_{i,c}.
\ee 
Using \eqref{sigma_l_vpc} and \eqref{XI_ic_sim}, the temporal covariance matrix is expressed as  
\begin{align}
\label{Sigma_z_temporal}
\boldsymbol{\Sigma}_{T,l} &= \sum_{v ,p,c} \sigma_{v,p,c}^l ~\big(\boldsymbol{\Psi}_{p,c}\odot \boldsymbol{\Upsilon}(v)) + \sigma_w^2\Vert\boldsymbol{\phi}_l^T\Vert_2^2\:\bI\:.
\end{align} 
If $b_l \in \mathcal{A}_{\phi}$ is the $l$th estimated AoA with its corresponding steering vector as the $l$th column in $\boldsymbol{\Phi}$ and $b_l = b_{i,\ell}$, then the coefficients $\sigma_{v,p,c}^l$ are nonzero for $v = v_{i,\ell},~ p = p_{i,\ell},~ c\in \mathcal{C}_i$. Thus, the filter $\boldsymbol{\Phi}$ makes the representation even sparser and more structured. The sparse representation of the temporal covariance matrix is more compactly written as 
\begin{align}
\label{Sigma_T_mat}
\text{vec}(\boldsymbol{\Sigma}_{T,l}) =\bM_T\boldsymbol{\sigma}_{T,l} + \sigma_w^2\Vert\boldsymbol{\phi}_l^T\Vert_2^2\:\text{vec}(\bI_M),
\end{align}
where $[\boldsymbol{\sigma}_{T,l}]_{\mu(v,p,c)}=\sigma_{v,p,c}^l$ is a $MNP\times 1$ sparse coefficient vector with non-negative entries and $[\bM_T]_{\mu(v,p,c)} = \text{vec}(\boldsymbol{\Psi}_{p,c}\odot \boldsymbol{\Upsilon}(v))$. 

\bibliographystyle{IEEEbib}

\end{document}